\documentclass[12pt]{amsart}
\usepackage{amsmath,amssymb,color,cite,colonequals,verbatim}
\usepackage{etoolbox}
\apptocmd{\sloppy}{\hbadness 10000\relax}{}{}

\numberwithin{equation}{section}

\newtheorem{thm}[equation]{Theorem}

\newtheorem{lem}[equation]{Lemma}
\newtheorem{cor}[equation]{Corollary}
\newtheorem{conj}[equation]{Conjecture}

\theoremstyle{definition}
\newtheorem{rmk}[equation]{Remark}

\newtheorem{defn}[equation]{Definition}

\newcommand{\F}{\mathbb{F}}
\newcommand{\bP}{\mathbb{P}}

\DeclareMathOperator{\charp}{char}

\DeclareMathOperator{\PSL}{PSL}

\DeclareMathOperator{\Aut}{Aut}
\DeclareMathOperator{\Hom}{Hom}
\DeclareMathOperator{\Gal}{Gal}
\DeclareMathOperator{\Frob}{Frob}

\newcommand{\mybar}[1]{#1\llap{$\overline{\phantom{\rm#1}}$}}
\newcommand{\abs}[1]{\lvert #1 \rvert}

\begin{document}

\title[Exceptionality and the Carlitz--Wan conjecture]{Exceptional extensions of local fields and the Carlitz--Wan conjecture}

\author{Zhiguo Ding}
\address{
  School of Mathematics and Statistics,
  Hunan Research Center of the Basic Discipline for Analytical Mathematics, HNP-LAMA,
  Central South University,
  Changsha 410083, China
}
\email{ding8191@csu.edu.cn}

\author{Wei Xiong}
\address{
  School of Mathematics,
  Hunan University,
  Changsha 410082, China
}
\email{weixiong@amss.ac.cn}

\author{Qifan Zhang}
\address{
  College of Mathematics,
  Sichuan University,
  Chengdu 610064, China
}
\email{zhangqifan@scu.edu.cn}

\date{\today}

\thanks
{The first author thanks Prof. Michael Zieve for introducing him to this subject, and in particular for communicating Lenstra's unpublished arguments.
The second author was supported by the Natural Science Foundation of Hunan Province of China (Grant No. 2020JJ4164).
All the authors thank Prof. Ye Tian for his help and encouragement, which in particular lead to this collaboration.}

\begin{abstract}
For any prime power $q$, a polynomial $f(X)\in\F_q[X]$ is ``exceptional'' if it induces bijections of $\F_{q^k}$ for infinitely many $k$; this condition is known to be equivalent to $f(X)$ inducing a bijection of $\F_{q^k}$ for at least one $k$ with $q^k\ge \deg(f)^4$.
In this paper, we introduce the notion of an ``exceptional'' extension of local fields of any characteristic, and show that if $f(X)\in\F_q[X]$ is exceptional in the classical sense then the field extension $\F_q(X)/\F_q(f(X))$ yields an exceptional local field extension upon passing to the completion at a degree-$1$ place.
We describe all exceptional local field extensions of degree coprime to the residue characteristic, determine the relationship between exceptionality of a local field extension and exceptionality of a subextension, and give various Galois-theoretic characterizations of exceptional local field extensions.
As a consequence, we obtain three new proofs, using quite different tools, of a theorem of Guralnick and M\"uller about ramification indices in exceptional maps between curves over $\F_q$. This theorem generalizes a result of Lenstra which subsumes earlier conjectures of Carlitz and Wan.
\end{abstract}

\dedicatory{Dedicated to Professor Daqing Wan in honor of his {\rm 60}th Birthday}

\maketitle

\section{Introduction}

A polynomial $f(X)\in\F_q[X]$ is called a \emph{permutation polynomial} if the map $\alpha \mapsto f(\alpha)$ defines a bijection on $\F_q$. Permutation polynomials arise in various contexts in math and engineering. Dickson determined the low-degree permutation polynomials over prime fields \cite{Dickson-prime} and over arbitrary finite fields \cite{Di}. See \cite{DZquad} for some recent progress on the study of permutation polynomials.

A rational function $f(X) \in \F_q(X)$ is called \emph{exceptional} if the map $\alpha \mapsto f(\alpha)$ is a bijection on $\bP^1(\F_{q^{k}})\colonequals  \F_{q^k}\cup\{\infty\}$ for infinitely many positive integers $k$. The study of exceptional rational functions was initiated in Dickson's 1896 thesis \cite{Di}, which investigated the special case of exceptional polynomials. See \cite{DZexc, DZAA, GM, GMS, GRZ, GTZ, GZ, LZ, Z} for some recent progress on exceptional rational functions.

One of the most famous problems in this subject is the Carlitz--Wan conjecture. Cavior \cite{Cavior-octic} first posed the ``Carlitz conjecture'' as a question. Cavior's advisor Carlitz mentioned this question later in a talk, after which the question became known as the Carlitz conjecture.

\begin{conj}[The Carlitz Conjecture] \label{carlitz}
There exists no even-degree exceptional polynomial over\/ $\F_q$ for any odd prime power $q$.
\end{conj}

The Carlitz conjecture was proved for polynomials of degree coprime to $q$ by Davenport and Lewis \cite{DL} and Hayes \cite{Hayes}, which in particular includes the case of degree $2^k$. Hayes \cite{Hayes} and Wan \cite{Wan-Carlitz} made progress on the Carlitz conjecture in low degrees, then the degree twice a prime case was resolved by Cohen \cite{Cohen-primitive} and Wan \cite{Wan-singularities}. The Carlitz conjecture in full generality was first proved by Fried, Guralnick and Saxl \cite{FGS} using the classification of finite simple groups.

Wan \cite{Wan-Carlitz-Wan_conjecture} posed and proved in some cases a generalization of the Carlitz conjecture, which became known as the Carlitz--Wan conjecture.

\begin{conj}[The Carlitz--Wan Conjecture] \label{carlitz-wan}
For any prime power $q$, every exceptional polynomial over\/ $\F_q$ has degree coprime to $q-1$.
\end{conj}

The Carlitz conjecture follows from the Carlitz--Wan conjecture in the even degree case. Note that Hayes \cite{Hayes} had previously proved the Carlitz--Wan conjecture for polynomials of degree coprime to $q$. The Carlitz--Wan conjecture follows from \cite{FGS} in the case of characteristic at least $5$, which depends on the classification of finite simple groups.
The Carlitz--Wan conjecture in full generality was first proved in 1994 by Lenstra, via an argument which does not rely on any difficult results. Lenstra never published his proof, perhaps in part because Cohen and Fried \cite{CF} published a version of Lenstra's proof soon after he sent it to them. However, Lenstra communicated several versions of his proof to the finite fields community through his emails and his lectures.

The Carlitz--Wan conjecture about exceptional polynomials can be generalized to exceptional covers of curves. By a curve over $\F_q$ we mean a dimension-$1$ geometrically integral separated scheme over $\F_q$ of finite type. Throughout this paper we assume in addition that all curves are nonsingular and projective. See \cite{H, St} for basic knowledge of algebraic geometry and in particular algebraic curves. We need the following notion of exceptionality for curves.

\begin{defn} \label{c-exc}
A finite morphism $f\colon X \to Y$ of curves over $\F_q$ is called an \emph{exceptional cover} if the diagonal is the only geometrically irreducible component defined over $\F_q$ of the fiber product $X\times_Y X$.
\end{defn}

The prototypical examples of exceptional covers come from isogenies of elliptic curves over $\F_q$, which are exceptional if and only if there exist no nonzero $\F_q$-rational points in the kernel. See \cite{GTZ} for the study of exceptional covers of varieties of possibly higher dimension.

Guralnick and M\"uller \cite[\S 8]{GM} gave two abstract group-theoretic proofs of a generalization of the Carlitz--Wan conjecture as follows.

\begin{thm} \label{gcw}
Assume $f\colon X \to Y$ is a finite morphism of curves over\/ $\F_q$. Suppose $f\colon X \to Y$ is an exceptional cover. Then the ramification index of $f$ at any\/ $\F_q$-rational point of $X$ is coprime to $q-1$.
\end{thm}

As will be explained later in this paper, the Carlitz--Wan conjecture follows from the special case of Theorem~\ref{gcw} when both curves $X$ and $Y$ have genus $0$ and the $\F_q$-rational point is totally ramified.

In this paper, we give three different proofs of Theorem~\ref{gcw}, which may be accessible and appreciated by readers with various backgrounds. More importantly, we develop the new theory of exceptionality for local fields, whose applications include the enlightening of the truth of the Carlitz--Wan conjecture and its generalization Theorem~\ref{gcw}.

The following is our new notion of exceptionality about local fields.

\begin{defn}
A finite separable extension $L/K$ of local fields is \emph{exceptional} if the following hold where $N$ is the Galois closure of $L/K$:
\begin{enumerate}
\item $L$ is totally ramified over $K$;
\item for any $\sigma \in \Gal(N/K) \setminus \Gal(N/L)$ the compositum $L.\sigma(L)$ of $L$ and $\sigma(L)$ is not totally ramified over $K$.
\end{enumerate}
\end{defn}

In this paper, we show that if $f(X)\in\F_q(X)$ is exceptional in the classical sense then the field extension $\F_q(X)/\F_q(f(X))$ yields an exceptional local field extension upon passing to the completion at a degree-one place.
We determine the relationship between exceptionality of a local field extension and exceptionality of a subextension, and give various Galois-theoretic characterizations of exceptional local field extensions.
Moreover, the following result describes all exceptional local field extensions of degree coprime to the residue characteristic, which sheds new insights on the conclusion of the Carlitz--Wan conjecture and its generalization Theorem~\ref{gcw}.

\begin{thm} \label{coprime}
Assume $L/K$ is a finite separable extension of local fields. Write $n\colonequals [L:K]$ and let\/ $\F_q$ be the residue field of $K$. Suppose $L/K$ is totally ramified and $\gcd(n,q)=1$. Then the following are equivalent:
\begin{enumerate}
\item $L/K$ is exceptional;
\item $\gcd(n,q-1)=1$;
\item $K$ contains no nontrivial $n$-th root of unity;
\item there exists no intermediate field $M$ of $L/K$ such that $M$ is a proper Galois extension over $K$.
\end{enumerate}
\end{thm}

As a consequence of our theory of exceptionality for local fields, we obtain three new proofs, using quite different tools, of the Guralnick--M\"uller result Theorem~\ref{gcw}, which rely on three quite different collections of tools, all of which differ from the group-theoretic arguments used by Guralnick--M\"uller.

The rest of this paper is organized as follows. In Section 2 we introduce the new notion of exceptionality for local fields and obtain some elementary properties on exceptional local field extensions. In Section 3 we first show the Carlitz-Wan conjecture follows from Theorem~\ref{gcw} and then we give two quick proofs of Theorem~\ref{gcw}. In Section 4 we establish in detail the new theory of exceptionality for local fields. We conclude this paper in Section 5 by giving the third proof of Theorem~\ref{gcw}.


\section{Elementary properties on exceptional local field extensions}

The main result in this section is the following elementary property on exceptional extensions of local fields, which will be used to give two quick proofs the generalization Theorem~\ref{gcw} of the Carlitz--Wan conjecture. See \cite{BGR, C, CaF, DZJA, Ne, S} for basic knowledge of local fields.

\begin{thm} \label{local}
Assume $L/K$ is an exceptional extension of local fields of degree $n$. Then $\gcd(n,q-1)=1$ where\/ $\F_q$ is the residue field of $K$.
\end{thm}

By a local field we mean a field which is complete with respect to a discrete valuation and has finite residue field. More explicitly, a local field is either a finite extension of the field $\mathbb{Q}_p$ of $p$-adic numbers or the field $\F_q((X))$ of formal Laurent series over a finite field $\F_q$. We recall the notion of exceptionality for local fields as follows.

\begin{defn}
A finite separable extension $L/K$ of local fields is \emph{exceptional} if the following hold where $N$ is the Galois closure of $L/K$:
\begin{enumerate}
\item $L$ is totally ramified over $K$;
\item for any $\sigma \in \Gal(N/K) \setminus \Gal(N/L)$ the compositum $L.\sigma(L)$ of $L$ and $\sigma(L)$ is not totally ramified over $K$.
\end{enumerate}
\end{defn}

\begin{rmk}
In particular, if $L/K$ is an exceptional extension of local fields, then the identity map is the only automorphism of $L$ which fixes each element of $K$.
For, any such automorphism can be extended to an element $\sigma$ of $\Gal(N/K)$ for which $\sigma(L)=L$, and then condition (2) implies $\sigma\in\Gal(N/L)$, so that $\sigma\vert_L$ is the identity on $L$.
Thus, an exceptional local field extension $L/K$ satisfies $\Aut_K(L)=\{\text{Id}_L\}$, which is the opposite extreme from $L/K$ being Galois.
\end{rmk}

\begin{rmk}
Note that elements in $\Gal(N/K) \setminus \Gal(N/L)$ correspond canonically to non-identity elements in $\Hom_K(L,N)$, which is the set of all $K$-algebra homomorphisms of $L$ into $N$. Thus exceptional local field extensions are the opposite of being Galois extensions, in the sense that there is no nontrivial $K$-algebra endomorphism of $L$.
\end{rmk}

We need the following result in the proof of Theorem~\ref{local}.

\begin{lem} \label{galois}
Any exceptional extension $L/K$ of local fields has no intermediate field $M$ such that $M$ is a proper Galois extension over $K$.
\end{lem}

\begin{proof}
Let $N$ be the Galois closure of $L/K$. Write $\mybar N, \mybar L, \mybar K$ for the residue fields of $N, L, K$ respectively, so that $\mybar L = \mybar K = \F_q$ for some power $q$ of the characteristic $p$ of $\mybar K$. By passing to the residue fields we have a canonical homomorphism $\rho \colon \Gal(N/K) \to \Gal(\mybar N/\mybar K)$ of groups, which is surjective since $\mybar K$ is perfect. Let $\theta$ be the $q$-th power map on $\mybar N$, so that $\theta$ is a generator of $\Gal(\mybar N/\mybar K)$. Write $\Frob(N/K')\colonequals  \rho^{-1}(\theta)\cap \Gal(N/K')$ for any intermediate field $K'$ of $N/K$.

First, we claim that $\Frob(N/\sigma(L))$ and $\Frob(N/\tau(L))$ are disjoint for any $\sigma,\tau \in \Gal(N/K)$ with $\sigma^{-1}\tau \notin \Gal(N/L)$. Henceforth we suppose $\sigma,\tau \in \Gal(N/K)$ with $\sigma^{-1}\tau \notin \Gal(N/L)$. Since $L/K$ is exceptional $L.(\sigma^{-1}\tau)(L)$ is not totally ramified over $K$, or equivalently $\sigma(L).\tau(L)$ is not totally ramified over $K$, which says the residue field of $\sigma(L).\tau(L)$ is a proper extension over $\F_q$. It follows that $\Gal(N/\sigma(L).\tau(L))$ is disjoint from $\rho^{-1}(\theta)$, which implies $\Frob(N/\sigma(L))$ and $\Frob(N/\tau(L))$ are disjoint. This concludes the proof of our claim.

Next, let us show $\Frob(N/K)$ is the union of $\Frob(N/\sigma(L))$ with $\sigma\in\Gal(N/K)$. Since the homomorphism $\rho \colon \Gal(N/K) \to \Gal(\mybar N/\mybar K)$ is surjective we know that $\abs{\Gal(N/K)} = \abs{\Frob(N/K)} \cdot \abs{\Gal(\mybar N/\mybar K)}$. Similarly we have $\abs{\Gal(N/\sigma(L))} = \abs{\Frob(N/\sigma(L))} \cdot \abs{\Gal(\mybar N/\mybar K)}$ for any $\sigma \in \Gal(N/K)$. Hence for any $\sigma \in \Gal(N/K)$ we have
\[
\frac{\abs{\Frob(N/K)}}{\abs{\Frob(N/\sigma(L))}} = \frac{\abs{\Gal(N/K)}}{\abs{\Gal(N/\sigma(L))}} = [\sigma(L):K] = [L:K].
\]
The claim says that $\Frob(N/\sigma(L))$ and $\Frob(N/\tau(L))$ are disjoint for any $\sigma,\tau \in \Gal(N/K)$ which lie in different left cosets of $\Gal(N/L)$ in $\Gal(N/K)$. It follows that
\[
\Frob(N/K) = \bigcup_{\sigma\in\Gal(N/K)} \Frob(N/\sigma(L))
\]
since there are $[L:K]$ left cosets of $\Gal(N/L)$ in $\Gal(N/K)$.

Suppose $M$ is an intermediate field of $L/K$ such that $M$ is Galois over $K$, so that $\Frob(N/M) = \Frob(N/K)$ since
\[
\Frob(N/K) = \bigcup_{\sigma\in\Gal(N/K)} \Frob(N/\sigma(L)) \subseteq \Frob(N/M) \subseteq \Frob(N/K).
\]
Since $M/K$ is totally ramified we have $\abs{\Gal(N/M)} = \abs{\Frob(N/M)} \cdot \abs{\Gal(\mybar N/\mybar K)}$, so that $\abs{\Gal(N/M)} = \abs{\Gal(N/K)}$, which implies that $\Gal(N/M)=\Gal(N/K)$, or equivalently $M=K$.
\end{proof}

Now we are ready to prove Theorem~\ref{local} as follows.

\begin{proof}[Proof of Theorem~\ref{local}]
Let $\mathcal{O}_K$ be the ring of integers in $K$ and $\mathfrak{m}_K$ be its maximal ideal, so that $\F_q$ is the residue field $\mybar K\colonequals  \mathcal{O}_K/\mathfrak{m}_K$ of $K$. Let $x$ and $z$ be uniformizers of the local fields $L$ and $K$, respectively.
Since $L/K$ is totally ramified we have $z = x^n \cdot \sum_{i=0}^{\infty} a_i x^i$ for some $a_i\in \mathcal{O}_K$ with $a_0\notin \mathfrak{m}_K$. We may assume $a_0=1$ by replacing $z$ with $a_0z$ if necessary. Write $n=mp^{\ell}$ for some $\ell\ge 0$ and $m>0$ with $p\nmid m$ where $p\colonequals \charp(\F_q)$.
Since $p\nmid m$ we have $\sum_{i=0}^{\infty} a_i x^i = \Bigl( \sum_{i=0}^{\infty} b_i x^i \Bigr)^m$ for some $b_i\in \mathcal{O}_K$ with $b_0=1$, thus $z = y^m$ where $y\colonequals  x^{p^{\ell}} \cdot \sum_{i=0}^{\infty} b_i x^i$. Write $M\colonequals  K(y)$, so that $[M:K]=m$ and $y$ is an uniformizer of $M$.

We claim that $K$ contains no nontrivial $m$-th root $\zeta$ of unity, which says that the residue field $\F_q$ of $K$ contains no nontrivial $m$-th root $\zeta$ of unity, or equivalently $\gcd(n,q-1)=1$. Suppose otherwise $K$ contains no nontrivial $m$-th root $\zeta$ of unity.
Thus $\zeta$ is a primitive $d$-th root of unity in $K$ for some divisor $d>1$ of $m$. It follows that $M'\colonequals K(y^{m/d})$ is Galois over $K$ of degree $d$. Hence $M'$ is an intermediate field of $L/K$ such that $M'$ is Galois over $K$. Since $L/K$ is exceptional by Lemma~\ref{galois} we know $M'=K$, or equivalently $d=1$, which is a contradiction.
\end{proof}


\section{Two quick proofs of the Carlitz--Wan conjecture}

In this section, first we show Theorem~\ref{gcw} implies the Carlitz--Wan conjecture, and then we give two quick proofs of Theorem~\ref{gcw}.

We will use the following classical result, see \cite{Co,GTZ}.

\begin{lem} \label{basics}
For any $f(X)\in\F_q(X)$ the following are equivalent:
\begin{enumerate}
\item $f(X)$ is exceptional over\/ $\F_q$;
\item irreducible factors of the numerator of $f(X)-f(Y)$ in\/ $\F_q[X,Y]$ which remain irreducible in\/ $\mybar \F_q[X,Y]$ are nonzero constants in\/ $\F_q$ times $X-Y$, where\/ $\mybar\F_q$ is the algebraic closure of\/ $\F_q$.
\end{enumerate}
\end{lem}

Next, we show Theorem~\ref{gcw} has the following consequence.

\begin{cor} \label{rational}
For any exceptional $f(X)\in\F_q(X)$ the ramification index of $f(X)$ at any point in\/ $\bP^1(\F_q)$ is coprime to $(q-1)$.
\end{cor}

\begin{proof}
Suppose $f(X)\in\F_q(X)$ is nonconstant, so that it induces a finite morphism $f\colon \bP^1\to \bP^1$ of curves. By Lemma~\ref{basics} the exceptionality of $f(X)$ says that $f\colon \bP^1\to \bP^1$ is an exceptional cover. By Theorem~\ref{gcw} it follows that the ramification index of $f(X)$ at any point in $\bP^1(\F_q)$ is coprime to $(q-1)$. This concludes the proof.
\end{proof}

Clearly, the Carlitz--Wan conjecture is Corollary~\ref{rational} in the special case of polynomials. Thus it remains only to show Theorem~\ref{gcw}.

In light of Theorem~\ref{local}, the proof of Theorem~\ref{gcw} is reduced to the following result, which says that the exceptionality of function fields implies the exceptionality of local fields at any rational place.

\begin{thm} \label{gtol}
Assume $f\colon X \to Y$ is a finite separable morphism of curves over\/ $\F_q$. Suppose $f\colon X\to Y$ is an exceptional cover. Then, for any\/ $\F_q$-point $u\in X(\F_q)$, the extension\/ $\F_q(X)_u/\F_q(Y)_v$ of local fields is exceptional, where\/ $\F_q(X)_u$ and\/ $\F_q(Y)_v$ are the completions of functions fields\/ $\F_q(X)$ and\/ $\F_q(Y)$ at the places $u$ and $v\colonequals f(u)$, respectively.
\end{thm}

\begin{proof}[Proof of Theorem~\ref{gcw}]
We may assume that the morphism $f\colon X \to Y$ of curves is finite separable. Let $u\in X(\F_q)$ be an $\F_q$-point of $X$, so that $v\colonequals f(u)\in Y(\F_q)$ is an $\F_q$-point of $Y$. By Theorem~\ref{gtol} we know the extension $\F_q(X)_u/\F_q(Y)_v$ of local fields is exceptional, where $\F_q(X)_u$ and $\F_q(Y)_v$ are the completions of $\F_q(X)$ and $\F_q(Y)$ at the places $u$ and $v$, respectively. By Theorem~\ref{local} it follows that $[\F_q(X)_u:\F_q(Y)_v]$ is coprime to $q-1$. This concludes the proof since $[\F_q(X)_u:\F_q(Y)_v]$ is equal to the ramification index of $f$ at $u$.
\end{proof}

Next, we give two different proofs of Theorem~\ref{gtol} as follows.


\subsection{The first proof}

Our first proof of Theorem~\ref{gtol} argues in terms of tensor products of function fields and their completions.

\begin{proof}[The first proof of Theorem~\ref{gtol}]
For ease of expression, let $L\colonequals \F_q(X)$ and $K\colonequals \F_q(Y)$ be the function fields of $X$ and $Y$, respectively. Thus $L/K$ is a finite separable extension of function fields over $\F_q$. Thus $L\otimes_K L$ is isomorphic as $K$-algebras to $\prod_{i=1}^r L_i$, where $L_i$ are finite separable extension fields over $K$ with $L_1=L$ which corresponds to the diagonal of the fiber product $X\times_Y X$.
The assumption that $f\colon X\to Y$ is an exceptional cover says that $L_1$ is the only geometrically irreducible among $L_i$ as function fields over $\F_q$.
We claim that any function field over $\F_q$ which has one rational place is geometrically irreducible. Indeed, this is equivalent to that any nonsingular projective curve $C$ over $\F_q$ which has one $\F_q$-point $P$ is geometrically irreducible. Suppose otherwise $C$ is not geometrically irreducible. Since the point $P$ is on the curve $C$, we know that $P$ is on one geometrically irreducible component of $C$, which by the $q$-th power map implies that $P$ is on each geometrically irreducible component of $C$. Since $C$ has more than one geometrically irreducible components, it follows that $P$ is a singular point on $C$, which contradicts with the assumption that $C$ is nonsingular.
Hence each function field $L_i$ with $i>1$ has no $\F_q$-rational places since it is not geometrically irreducible. Moreover, each $L_i$ can be viewed as a finite separable extension over $L$ via the right $L$-module structure of $L\otimes_K L$.
Let $L_u$ and $K_v$ be the completions of $L$ and $K$ at places $u$ and $v$, respectively. Since both $u$ and $v$ are $\F_q$-rational, the extension $L_u/K_v$ is a totally ramified extension of local fields with $\F_q$ as residue fields. Moreover, we have $L\otimes_K K_v = \prod_{j=1}^s L_{u_j} = L_u \times \prod_{j=2}^s L_{u_j}$, where $u_1,\dots,u_s$ are all places of $L$ lying over $v$ with $u_1=u$.

Let us compute the tensor product $L\otimes_K L_u$ in two different ways. On the one hand, by viewing the $K$-algebra structure of $L_u$ via the composition $K\to L\to L_u$, we can compute
\begin{align*}
L\otimes_K L_u & \cong L\otimes_K \bigl( L\otimes_L L_u \bigr) \\
& \cong \bigl( L\otimes_K L \bigr) \otimes_L L_u  \\
& \cong \Bigl( \prod_{i=1}^r L_i \Bigr) \otimes_L L_u \\
& \cong  \prod_{i=1}^r \bigl( L_i\otimes_L L_u \bigr) \\
& \cong L_u \times \prod_{i=2}^r \bigl( L_i\otimes_L L_u \bigr).
\end{align*}
Since each $L_i$ with $i\ge 2$ has no $\F_q$-rational places, it follows that $\prod_{i=2}^r \bigl( L_i\otimes_L L_u \bigr)$ is a product of local fields whose residue fields are proper extensions over $\F_q$. Thus the local field $L_u$ is the only factor of the tensor product $L\otimes_K L_u$ whose residue field is $\F_q$. On the other hand, by viewing the $K$-algebra structure of $L_u$ via the composition $K\to K_v\to L_u$, we can compute
\begin{align*}
L\otimes_K L_u & \cong L\otimes_K \bigl( K_v\otimes_{K_v} L_u \bigr) \\
& \cong \bigl( L\otimes_K K_v \bigr) \otimes_{K_v} L_u \\
& \cong \Bigl( \prod_{j=1}^s L_{u_j} \Bigr) \otimes_{K_v} L_u \\
& \cong \prod_{j=1}^s \bigl( L_{u_j}\otimes_{K_v} L_u \bigr) \\
& \cong \bigl( L_u\otimes_{K_v} L_u \bigr) \times \prod_{j=2}^s \bigl( L_{u_j}\otimes_{K_v} L_u \bigr).
\end{align*}
By comparing the above two expressions of $L\otimes_K L_u$ it follows that the local field $L_u$ is the only factor of $L_u\otimes_{K_v} L_u$ whose residue field is $\F_q$. Therefore, the finite separable extension $L_u/K_v$ of local fields is totally ramified, while the compositum of any two distinct conjugates of $L_u$ are not totally ramified over $K_v$. Thus $L_u/K_v$ is an exceptional extension of local fields, which concludes the proof.
\end{proof}


\subsection{The second proof}

Our second proof of Theorem~\ref{gtol} uses some basic properties on Galois groups of extensions of function fields.

For ease of expression we introduce the following notion, which is the motivating instance of the general notion of exceptional field extensions introduced in Lenstra's 1995 talk at Glasgow.

\begin{defn}
Assume $L/K$ is a finite separable extension such that $\F_q\subseteq K$. We say $L/K$ has the \emph{Lenstra-property} if the following hold where $N$ is the Galois closure of $L/K$:
\begin{enumerate}
\item $\F_q$ is algebraically closed in $L$;
\item $\F_q$ is not algebraically closed in the compositum $L.\sigma(L)$ of $L$ and $\sigma(L)$ for any $\sigma \in \Gal(N/K) \setminus \Gal(N/L)$.
\end{enumerate}
\end{defn}

Our second proof of Theorem~\ref{gtol} proceeds in two steps. First, we show exceptional extensions of function fields have the Lenstra-property.

\begin{lem} \label{lenstra}
Assume $f\colon X \to Y$ is a finite separable morphism of curves over\/ $\F_q$. Suppose $f\colon X\to Y$ is an exceptional cover. Then the function field extension\/ $\F_q(X)/\F_q(Y)$ has the Lenstra-property.
\end{lem}

\begin{proof}
Let $\Omega$ be the Galois closure of the finite separable extension $\F_q(X)/\F_q(Y)$ of function fields over $\F_q$, and let $\F_{q^{\ell}}$ be the algebraically closure of $\F_q$ in $\Omega$.
Write $A\colonequals \Gal(\Omega/\F_q(Y))$, $G\colonequals \Gal(\Omega/\F_{q^{\ell}}(Y))$, $A_1\colonequals \Gal(\Omega/\F_q(X))$, $G_1\colonequals \Gal(\Omega/\F_{q^{\ell}}(X))$, so that $A$ acts naturally on the set $\mathcal{S}$ of the left cosets of $A_1$ in $A$.
Moreover, $G_1$ is a normal subgroup of $A_1$ and the quotient $A_1/G_1$ is canonically isomorphic to $\Gal(\F_{q^{\ell}}/\F_q)$. The exceptionality of $f\colon X\to Y$ says that $\{A_1\}$ is the only common orbit of $\mathcal{S}$ under the actions by $A_1$ and $G_1$.

Note that $\F_q$ is algebraically closed in $\F_q(X)$ since $X$ is geometrically integral over $\F_q$. Suppose $\sigma \in \Gal(\Omega/\F_q(Y)) \setminus \Gal(\Omega/\F_q(X))$. It is enough to show that $\F_q$ is not algebraically closed in the compositum $\F_q(X).\sigma(\F_q(X))$.
Write $H\colonequals \Gal(\Omega/\F_q(X).\sigma(\F_q(X)))$, so that $H$ is a subgroup of $A_1$. Since $\Gal(\F_{q^{\ell}}/\F_q)$ is cyclic, the image $\rho(H)$ of $H$ under the composition $\rho\colon A_1 \to A_1/G_1 \to \Gal(\F_{q^{\ell}}/\F_q)$ is cyclic. Pick $\tau\in H$ such that $\rho(\tau)$ is a generator of $\rho(H)$.
Thus $\tau\in A_1$ and $\tau$ fixes two distinct elements $A_1$ and $\sigma A_1$ in $\mathcal{S}$. Since $\{A_1\}$ is the only common orbit of $\mathcal{S}$ under the actions by $A_1$ and $G_1$, it follows that the image $\tau$ in $A_1/G_1$ cannot be a generator of $A_1/G_1$, or equivalently $\rho(H) = \Gal(\F_{q^{\ell}}/\F_{q^m})$ for some divisor $m>1$ of $\ell$.
It follows that $\F_{q^m}$ is contained in $\F_q(X).\sigma(\F_q(X))$, which implies that $\F_q$ is not algebraically closed in $\F_q(X).\sigma(\F_q(X))$. It concludes the proof.
\end{proof}

Next, we show the Lenstra-property is preserved by the completion.

\begin{lem} \label{lenstra2}
Assume $L/K$ is a finite separable extension of function fields over\/ $\F_q$. Suppose $L/K$ satisfies the Lenstra-property and $u$ is any\/ $\F_q$-rational place of $L$. Let $v$ be the place of $K$ lying under $u$, and let $L_u$ and $K_v$ be the completions of $L$ and $K$ at $u$ and $v$, respectively. Then the extension $L_u/K_v$ of completions has the Lenstra-property, or equivalently $L_u/K_v$ is an exceptional extension of local fields.
\end{lem}

\begin{proof}
Since $u$ is $\F_q$-rational, $\F_q$ is algebraically closed in $L_u$. Let $N$ be the Galois closure of $L/K$, and let $\widehat N$ be the Galois closure of $L_u/K_v$.
For any $\sigma \in \Gal(\widehat N/K_v) \setminus \Gal(\widehat N/L_u)$ we have $L_u.\sigma(L_u) = K_v.(L.\tau(L))$ for some $\tau \in \Gal(N/K) \setminus \Gal(N/L)$, so that $\F_q$ is not algebraically closed in $L_u.\sigma(L_u)$.
Thus $L_u/K_v$ has the Lenstra-property. It is easy to show that $L_u/K_v$ has the Lenstra-property if and only if $L_u/K_v$ is an exceptional extension of local fields. This concludes the proof.
\end{proof}

Theorem~\ref{gtol} follows from the combination of Lemmas~\ref{lenstra} and \ref{lenstra2}.

\begin{rmk}
The converse of Lemma~\ref{lenstra} is also true, which can be shown similarly. But it is not needed in this proof of Theorem~\ref{gcw}.
\end{rmk}


\section{Theory of exceptionality for local fields}

Suppose $L/K$ is a finite separable extension of local fields. Let $N$ be the Galois closure of $L/K$. For any intermediate field $M$ of $N/K$, let $M'$ be the maximal unramified extension over $M$ inside $N$, and let $\mybar M$ be the residue field of $M$.
In other words, $\Gal(N/M')$ is the kernel of the canonical surjective homomorphism $\Gal(N/M) \to \Gal(\mybar N / \mybar M)$. Equivalently, $M'$ is obtained from $M$ by adjoining all roots of unity in $N$ of order coprime to the characteristic of the residue field $\mybar N$.

\begin{defn}
The groups $A\colonequals \Gal(N/K)$ and $G\colonequals \Gal(N/K')$ are called the \emph{arithmetic monodromy group} and the \emph{geometric monodromy group} of the finite separable extension $L/K$ of local fields, respectively.
\end{defn}

Write $A_1\colonequals \Gal(N/L)$ and $G_1\colonequals \Gal(N/L')$, so that $A_1$ and $G_1$ are subgroups of $A$ and $G$, respectively. Moreover, $G$ is a normal subgroup of $A$ and the quotient $A/G$ is canonically isomorphic to $\Gal(\mybar N/\mybar K)$. Similarly, $G_1$ is a normal subgroup of $A_1$ and the quotient $A_1/G_1$ is canonically isomorphic to $\Gal(\mybar N/\mybar L)$. In particular, both $A/G$ and $A_1/G_1$ are finite cyclic groups.

Let $\mathcal{S}\colonequals  A/A_1$ be the set of all left cosets of $A_1$ in $A$, so that $A$ acts transitively on $\mathcal{S}$. There is a canonical bijection between $\mathcal{S}$ and the set $\Hom_K(L,N)$ of all $K$-algebra homomorphisms from $L$ to $N$. Since $G\cap A_1 = G_1$ we have the embedding $A_1/G_1 \to A/G$ as groups and the embedding $G/G_1 \to A/A_1$ as sets of left cosets. Moreover, we have

\begin{lem} \label{t-ram}
Suppose $L/K$ is a finite separable extension of local fields. Then the following statements are equivalent:
\begin{enumerate}
\item $L/K$ is totally ramified;
\item $[L:K] = [L':K']$;
\item $[L':L] = [K':K]$;
\item the embedding $A_1/G_1 \to A/G$ is an isomorphism;
\item the embedding $G/G_1 \to A/A_1$ is a bijection;
\item $G$ acts transitively on $\mathcal{S}$.
\end{enumerate}
\end{lem}

\begin{proof}
For any two intermediate fields $M_1$ and $M_2$ of $N/K$ such that $M_1\subseteq M_2$, let $e(M_2/M_1)$ and $f(M_2/M_1)$ be the ramification index and the residue index of the extension $M_2/M_1$ of local fields, respectively. Since $e(L'/L)=1=e(K'/K)$ we have $e(L/K)=e(L'/K')$.
Since $f(L'/K')=1$ it follows that $L/K$ is totally ramified if and only if $[L:K] = [L':K']$, or equivalently $[L':L] = [K':K]$. Moreover, by Galois theory we have $[K':K] = [A:G]$, and $[L':L] = [A_1:G_1]$, $[L:K] = [A:A_1]$, and $[L':K'] = [G:G_1]$. Thus $[L':L] = [K':K]$ says that the embedding $A_1/G_1 \to A/G$ is an isomorphism as groups.
Similarly, $[L:K] = [L':K']$ says that the embedding $G/G_1 \to A/A_1$ is a bijection as sets of left cosets. Finally, since $G/G_1$ is one $G$-orbit in $A/A_1$ it follows that the embedding $G/G_1 \to A/A_1$ is a bijection if and only if $G$ acts also transitively on $\mathcal{S}\colonequals A/A_1$.
\end{proof}

Henceforth we assume in addition that $L/K$ is totally ramified, so that we have the triple $(A,G,\mathcal{S})$ in which $G$ is a normal subgroup of a finite group $A$ and $\mathcal{S}$ is a finite set with an $A$-action such that the quotient $A/G$ is cyclic and $G$ acts transitively on $\mathcal{S}$. Thus we can use the following two lemmas, see \cite[Lemmas~4.2 and 4.3]{GTZ}.

\begin{lem} \label{burnside}
Let $H_1$ be a finite group acting on a finite set $\mathcal{T}$, let $H_2$ be a normal subgroup of $H_1$ with $H_1/H_2$ cyclic, and let $\sigma H_2$ be a generator of $H_1/H_2$. Then the number of $H_1$-orbits on $\mathcal{T}$ which are also $H_2$-orbits equals
\[
\frac{1}{\abs{H_2}} \sum_{h\in\sigma H_2} \abs{\mathcal{T}^h},
\]
where $\mathcal{T}^h$ denotes the set of fixed points of $h$ in $\mathcal{T}$.
\end{lem}

\begin{lem} \label{count}
Suppose $H_2$ is a normal subgroup of a finite group $H_1$ and $\mathcal{T}$ is a finite set with an $H_1$-action such that the quotient $H_1/H_2$ is cyclic and $H_2$ acts transitively on $\mathcal{T}$. Then the following are equivalent:
\begin{enumerate}
\item the diagonal is the unique common orbit of $H_1$ and $H_2$ in $\mathcal{T}\times \mathcal{T}$;
\item any $\sigma\in H_1$ with $H_1=\langle H_2,\sigma\rangle$ has a unique fixed point in $\mathcal{T}$;
\item any $\sigma\in H_1$ with $H_1=\langle H_2,\sigma\rangle$ has at most one fixed point in $\mathcal{T}$;
\item any $\sigma\in H_1$ with $H_1=\langle H_2,\sigma\rangle$ has at least one fixed point in $\mathcal{T}$.
\end{enumerate}
\end{lem}

Let $\rho \colon \Gal(N/K) \to \Gal(\mybar N/\mybar K)$ be the canonical surjective group homomorphism obtained by passing to the residue fields. Thus we have $G=\ker(\rho)$ and $G_1=\ker(\rho) \cap A_1$. Suppose $\mybar K=\F_q$ for some power of a prime $p$, and let $\theta\in\Gal(\mybar N/\mybar K)$ be the canonical generator which acts on $\mybar N$ as the $q$-the power map.
For any intermediate field $M$ of $L/K$, write $\Frob(N/M)\colonequals  \rho^{-1}(\theta) \cap \Gal(N/M)$, which consists of all $\rho$-preimages of $\theta$ in $\Gal(N/M)$. So $\Frob(N/K)$ is a canonical generator of the quotient group $A/G$. Similarly, since $L/K$ is totally ramified, $\Frob(N/L)$ is a canonical generator of the quotient group $A_1/G_1$.

With the notations induced above, we are ready to characterize the exceptionality for local fields in various ways as follows.

\begin{thm} \label{nt-ram}
Assume $L/K$ is a finite separable extension of local fields. If $L/K$ is totally ramified then the following are equivalent:
\begin{enumerate}
\item $L/K$ is exceptional;
\item $\{A_1\}$ is the unique common orbit of $A_1$ and $G_1$ in $\mathcal{S}$;
\item the diagonal is the unique common orbit of $A$ and $G$ in $\mathcal{S}\times \mathcal{S}$;
\item $\Frob(N/L)\cap\Frob(N/\sigma(L)) = \emptyset$ for any $\sigma\in A\setminus A_1$;
\item $\Frob(N/K) = \bigcup_{\sigma\in A} \Frob(N/\sigma(L))$ holds;
\item any $\sigma\in A$ with $A=\langle G,\sigma\rangle$ has a unique fixed point in $\mathcal{S}$;
\item any $\sigma\in A$ with $A=\langle G,\sigma\rangle$ has at most one fixed point in $\mathcal{S}$;
\item any $\sigma\in A$ with $A=\langle G,\sigma\rangle$ has at least one fixed point in $\mathcal{S}$;
\item any $\sigma\in \Frob(N/K)$ has a unique fixed point in $\mathcal{S}$;
\item any $\sigma\in \Frob(N/K)$ has at most one fixed point in $\mathcal{S}$;
\item any $\sigma\in \Frob(N/K)$ has at least one fixed point in $\mathcal{S}$;
\item any $\sigma\in A_1$ with $A_1=\langle G_1,\sigma\rangle$ has a unique fixed point in $\mathcal{S}$;
\item any $\sigma\in A_1$ with $A_1=\langle G_1,\sigma\rangle$ has at most one fixed point in $\mathcal{S}$;
\item any $\sigma\in \Frob(N/L)$ has a unique fixed point in $\mathcal{S}$;
\item any $\sigma\in \Frob(N/L)$ has at most one fixed point in $\mathcal{S}$.
\end{enumerate}
\end{thm}

\begin{proof}
Since $L/K$ is totally ramified, it follows from Lemma~\ref{t-ram} that $G$ acts transitively on $\mathcal{S}$, so that by Lemma~\ref{count} items (3) and (6)--(8) are equivalent. Since $\Frob(N/K)$ is a generator of $A/G$, by applying Lemma~\ref{burnside} to $(A,G,\mathcal{S}\times\mathcal{S})$ it follows that items (3) and (9)--(11) are equivalent.
Similarly, by applying Lemma~\ref{burnside} to $(A_1,G_1,\mathcal{S})$ it follows that items (2) and (12)--(15) are equivalent. It is standard to verify the equivalence between items (2) and (3). Hence items (2)--(3) and (6)--(12) are all equivalent to one another.

Next, we show the equivalence between items (1) and (2). Henceforth we suppose $\sigma\in A\setminus A_1$. Let $M\colonequals L.\sigma(L)$ be the compositume of $L$ and $\sigma(L)$. Since $e(M'/M)=1=e(K'/K)$, we have $e(M/K)=e(M'/K')$. Since $f(M'/K')=1$, it follows that $M/K$ is totally ramified if and only if $[M:K] = [M':K']$.
Note $[M:K]$ and $[M':K']$ are the lengths of the orbits containing $\sigma A_1$ under the actions $A_1$ and $G_1$, respectively. Hence $[M:K] = [M':K']$ holds if and only if the $A_1$-orbit of $\sigma A_1$ is the same as the $G_1$-orbit of $\sigma A_1$. Thus items (1) and (2) are equivalent.

Next, let us show the equivalence between items (1) and (4). Pick $\sigma\in A\setminus A_1$ and let $M\colonequals  L.\sigma(L)$ be the compositum of $L$ and $\sigma(L)$. So $M$ is not totally ramified over $K$ if and only if $\mybar M$ is larger than $\F_q\colonequals \mybar K$.
Since $\rho \colon \Gal(N/M) \to \Gal(\mybar N/\mybar M)$ is surjective, it follows that $\F_q\subsetneq \mybar M$ if and only if $\Frob(N/M) \colonequals \rho^{-1}(\theta) \cap \Gal(N/M)$ is empty, or equivalently $\Frob(N/L)\cap\Frob(N/\sigma(L)) = \emptyset$.
Hence $L.\sigma(L)$ is not totally ramified over $K$ if and only if $\Frob(N/L)\cap\Frob(N/\sigma(L)) = \emptyset$ for any $\sigma\in A\setminus A_1$, which shows the equivalence of (1) and (4).

It remains to show only that (4) and (5) are equivalent. First, we claim item (4) holds if and only if $\Frob(N/\tau(L))\cap\Frob(N/\sigma(L))=\emptyset$ for any $\tau,\sigma\in A$ with $\tau A_1 \ne \sigma A_1$.
It is enough to show the ``only if'' part of this claim, since the ``if'' part follows easily by taking $\tau$ as the identity map on $N$. Henceforth we suppose item (4) holds. Pick any $\tau,\sigma\in A$ with $\tau A_1 \ne \sigma A_1$, so that $\tau^{-1}\sigma\in A\setminus A_1$.
It follows from item (4) that $\Frob(N/L)\cap\Frob(N/(\tau^{-1}\sigma)(L)) = \emptyset$, or equivalently $\rho^{-1}(\theta)\cap\Gal(N/L.(\tau^{-1}\sigma)(L)) = \emptyset$, which by conjugation with $\tau$ says $\rho^{-1}(\theta)\cap\Gal(N/\tau(L).\sigma(L)) = \emptyset$, which is equivalent to the desired $\Frob(N/\tau(L))\cap\Frob(N/\sigma(L)) = \emptyset$.
Next, we claim item (5) holds if and only if the above same property holds, which says $\Frob(N/\tau(L))\cap\Frob(N/\sigma(L))=\emptyset$ for any $\tau,\sigma\in A$ with $\tau A_1 \ne \sigma A_1$. Since the homomorphism $\rho \colon \Gal(N/K) \to \Gal(\mybar N/\mybar K)$ is surjective we know $\abs{\Gal(N/K)} = \abs{\Frob(N/K)} \cdot \abs{\Gal(\mybar N/\mybar K)}$.
Similarly, since $L/K$ is totally ramified $\abs{\Gal(N/\sigma(L))} = \abs{\Frob(N/\sigma(L))} \cdot \abs{\Gal(\mybar N/\mybar K)}$ for any $\sigma \in A$. Hence for any $\sigma \in A$ we have
\[
\frac{\abs{\Frob(N/K)}}{\abs{\Frob(N/\sigma(L))}} = \frac{\abs{\Gal(N/K)}}{\abs{\Gal(N/\sigma(L))}} = [\sigma(L):K] = [L:K].
\]
Hence our second claim follows by counting the sizes. Thus items (4) and (5) are equivalent, which concludes the proof.
\end{proof}

Roughly speaking, our next result says that the exceptionality for local fields is preserved well under the composition in a perfect way.

\begin{thm} \label{subext}
Suppose $L/K$ is a finite separable extension of local fields. Then for any intermediate field $M$ of $L/K$ we have that $L/K$ is exceptional if and only if both $L/M$ and $M/K$ are exceptional.
\end{thm}

\begin{proof}
Suppose $L/K$ is a finite separable extension of local fields and $M$ is an intermediate field of $L/K$. It is obvious that $L/K$ is totally ramified if and only if both $L/M$ and $M/K$ are totally ramified. Since the exceptionality implies the total ramification, we may assume in addition that $L/K$ is totally ramified.

First, we show the ``if'' part. Suppose $L/K$ is not exceptional, so that $L.\sigma(L)$ is totally ramified over $K$ for some $\sigma\in A\setminus A_1$, or equivalently $L.\sigma(L)$ is totally ramified over $K$ for some $\sigma\in \Hom_K(L,N)$ which is not the identity map on $L$.
If the restriction $\sigma|_M$ of $\sigma$ on $M$ is the identity map, then $\sigma$ is a non-identity map in $\Hom_M(L,N)$, so that $L/M$ is not exceptional since $L.\sigma(L)$ is totally ramified over $M$.
If $\sigma|_M$ is not the identity map on $M$, then $\sigma|_M$ is a non-identity map in $\Hom_K(M,N)$ such that $M.\sigma|_M(M)$ is totally ramified over $K$, which implies that $M/K$ is not exceptional.

Next, we show the ``only if'' part. Suppose $L/K$ is exceptional, which by item (9) of Theorem~\ref{nt-ram} says that any $\sigma\in \Frob(N/K)$ has a unique fixed point in $\mathcal{S}$, or equivalently any $\sigma\in \Gal(N/K)$ which induces the $q$-th power map on $\mybar N$ has a unique fixed point in $\Hom_K(L,N)$, where $\F_q\colonequals \mybar K$ is the residue field of $K$.
Let $N_1$ and $N_2$ be the Galois closures in $N$ of $L/M$ and $M/K$, respectively. In order to show $L/M$ is exceptional, by item (10) of Theorem~\ref{nt-ram} it is enough to show that any element in $\Gal(N_1/M)$ which induces the $q$-th power map on $\mybar N_1$ has at most a fixed point in $\Hom_M(L,N_1)$.
Similarly, in order to show $M/K$ is exceptional, by item (11) of Theorem~\ref{nt-ram} it is enough to show that any element in $\Gal(N_2/K)$ which induces the $q$-th power map on $\mybar N_2$ has at least a fixed point in $\Hom_K(M,N_2)$.

We claim that any element in $\Gal(N_1/M)$ which induces the $q$-th power map on $\mybar N_1$ can be lifted to some element in $\Gal(N/M)$ which induces the $q$-th power map on $\mybar N$. Suppose $\sigma\in \Gal(N_1/M)$ which induces the $q$-th power map on $\mybar N_1$.
We can lift $\sigma$ to some element in $\Gal(N/M)$, which is also denoted by $\sigma$ by abuse of language. Pick $\tau\in \Gal(N/K)$ which induces the $q$-th power map on $\mybar N$. Thus both $\sigma$ and $\tau$ induce the $q$-th power map on $\mybar N_1$.
Since the homomorphism $\Gal(N/N_1) \to \Gal(\mybar N/\mybar N_1)$ is surjective, there exists $\widetilde\sigma\in \Gal(N/M)$ which extends $\sigma$ and induces the $q$-th power map on $\mybar N$. By the same argument as above, we can show that any element in $\Gal(N_2/K)$ which induces the $q$-th power map on $\mybar N_2$ can be lifted to some element in $\Gal(N/K)$ which induces the $q$-th power map on $\mybar N$.

Now we are ready to show both $L/M$ and $M/K$ are exceptional. First we show $L/M$ is exceptional. Suppose $\sigma\in\Gal(N_1/M)$ which induces the $q$-th power map on $\mybar N_1$. By the above claim $\sigma$ can be lifted to some $\widetilde\sigma\in\Gal(N/M)$ which induces the $q$-th power map on $\mybar N$.
Since $L/K$ is exceptional it follows that $\widetilde\sigma$ has a unique fixed point in $\Hom_K(L,N)$, which implies that $\sigma$ has at most a fixed point in $\Hom_M(L,N_1)$ since $\Hom_M(L,N_1)$ is a subset of $\Hom_K(L,N)$. Thus $L/M$ is exceptional.
Next let us show $M/K$ is exceptional. Suppose $\sigma\in\Gal(N_2/K)$ which induces the $q$-th power map on $\mybar N_2$, so that $\sigma$ can be lifted to some $\widetilde\sigma\in\Gal(N/K)$ which induces the $q$-th power map on $\mybar N$.
Since $L/K$ is exceptional it follows that $\widetilde\sigma$ has a unique fixed point in $\Hom_K(L,N)$, which implies that $\sigma$ has at least a fixed point in $\Hom_K(M,N_2)$ since there is a map $\Hom_K(L,N) \to \Hom_K(M,N_2)$ obtained by the restriction. Thus $M/K$ is exceptional.
\end{proof}

Now we are ready to show Theorem~\ref{coprime}, which characterizes exceptional local field extensions of degree coprime to the characteristic of the residue fields.

\begin{proof}[Proof of Theorem~\ref{coprime}]
Note that roots of unity in a local field of order coprime to the characteristic of the residue field are the image of the Teichm\"uller character.
Since $\gcd(n,q)=1$ it follows that $K$ has some nontrivial $n$-th root of unity if and only if $\F_q$ has some nontrivial $n$-th root of unity, or equivalently $\gcd(n,q-1)\ne 1$. Thus item (2) and (3) are equivalent. It remains to show items (1), (3), (4) are equivalent.

Let $\mathcal{O}_K$ be the ring of integers in $K$ and $\mathfrak{m}_K$ be its maximal ideal, so that $\F_q$ is the residue field $\mybar K\colonequals  \mathcal{O}_K/\mathfrak{m}_K$ of $K$. Let $x$ and $z$ be uniformizers of the local fields $L$ and $K$, respectively.
Since $L/K$ is totally ramified we have $z = x^n \cdot \sum_{i=0}^{\infty} a_i x^i$ for some $a_i\in \mathcal{O}_K$ with $a_0\notin \mathfrak{m}_K$. We may assume in addition that $a_0=1$ by replacing $z$ with $a_0z$ if necessary. Since $\gcd(n,q)=1$ we have $\sum_{i=0}^{\infty} a_i x^i = \Bigl( \sum_{i=0}^{\infty} b_i x^i \Bigr)^n$ for some $b_i\in \mathcal{O}_K$ with $b_0=1$.
Write $y\colonequals  x \cdot \sum_{i=0}^{\infty} b_i x^i$, so that $z = y^n$ and $y$ is an uniformizer of $M$ with $L\colonequals  K(y)$. Since $y$ is a generator of $L$ over $K$, it follows that $X^n-z\in K[X]$ the minimal polynomial of $y$ over $K$. Let $N$ be the Galois closure of $L/K$, so that
\[
X^n-z = X^n-y^n = \prod_{i=0}^{n-1} (X-\zeta^iy)
\]
where $\zeta$ is a primitive $n$-th root of unity in $N$.

Now we are ready to show that items (1) and (3) are equivalent. For any $\sigma\in\Gal(N/K)$ it is easy to check that $\sigma\notin\Gal(N/L)$ if and only if $\sigma(L) = K(\zeta^i y)$ for some integer $i$ with $1\le i\le n-1$. Suppose $\sigma\in\Gal(N/K)\setminus\Gal(N/L)$, so that $\sigma(L) = K(\zeta^i y)$ for some $1\le i\le n-1$.
Since $L.\sigma(L) = K(\zeta^i,y)$, it follows that $L.\sigma(L)$ is totally ramified over $K$ if and only if $K(\zeta^i,y)$ and $K$ contain the same roots of unity of order coprime to $q$, or equivalently $\zeta^i$ lies in $K$.

It remains only to show the equivalence between items (1) and (4). First we suppose $L/K$ is not exceptional, so that $K$ contains some nontrivial $n$-th root of unity. So there exists some divisor $m>1$ of $n$ such that $\zeta^{n/m}\in K$.
It follows that $M\colonequals K(y^{n/m})$ is an intermediate field of $L/K$ such that $M/K$ is Galois of degree $m>1$. Next we suppose there exists some intermediate field $M$ of $L/K$ such that $M$ is a proper Galois extension over $K$.
Since $[M:K]>1$ we can pick some non-identity $\sigma\in \Gal(M/K)$, so that $M.\sigma(M)=M$ which is totally ramified over $K$. Hence $M/K$ is not exceptional, which by Theorem~\ref{subext} implies $L/K$ is not exceptional. This concludes the proof.
\end{proof}


\section{Third proof of the Carlitz--Wan conjecture}

In this section we give the third proof of Theorem~\ref{gcw}, which is the generalization of the Carlitz--Wan conjecture to curves. This proof uses more group theory in terms of monodromy groups, decomposition groups, and inertia groups.
It requires more knowledge of exceptional local field extensions, which has been established in Section 4. The advantage of this proof is that it is more insightful in the sense that it sheds more lights on the truth of the Carlitz--Wan conjecture.

Let us restate as follows the notion of exceptional covers introduced in Definition~\ref{c-exc} in the language of function fields. Since the part of pure inseparability will not make any essential difference for our purposes, we may assume the finite morphism $f\colon X \to Y$ of curves over $\F_q$ is separable, so that the extension $\F_q(X)/\F_q(Y)$ of function fields over $\F_q$ is finite separable. Let $\Omega$ be the Galois closure of $\F_q(X)/\F_q(Y)$, and let $\F_{q^{\ell}}$ be the algebraically closure of $\F_q$ in $\Omega$.

\begin{defn}
The Galois groups $A\colonequals \Gal(\Omega/\F_q(Y))$ and $G\colonequals \Gal(\Omega/\F_{q^{\ell}}(Y))$ are called the \emph{arithmetic monodromy group} and the \emph{geometric monodromy group} of $\F_q(X)/\F_q(Y)$, respectively.
\end{defn}

Thus $G$ is a normal subgroup of $A$ with quotient $A/G$ canonically isomorphic to $\Gal(\F_{q^{\ell}}/\F_q)$, so that $A/G$ is finite cyclic of order $\ell$. Write $A_1\colonequals \Gal(\Omega/\F_q(X))$ and $G_1\colonequals \Gal(\Omega/\F_{q^{\ell}}(X))$.
Let $\mathcal{S}$ be the set of all left cosets of $A_1$ in $A$, which may be viewed canonically as the set of all $\F_q(Y)$-algebra homomorphisms of $\F_q(X)$ into $\Omega$. Thus $A$ acts naturally on $\mathcal{S}$ and $G$ acts transitively on $\mathcal{S}$.

The following result is a standard fact on exceptional covers, which follows from the combination of \cite[Lemmas~4.1 and 4.3]{GTZ}.

\begin{lem} \label{standard}
Suppose $f\colon X \to Y$ is a finite separable morphism of curves over\/ $\F_q$. Then the following properties are equivalent:
\begin{enumerate}
\item $f\colon X \to Y$ is an exceptional cover;
\item $\{A_1\}$ is the unique common orbit of $A_1$ and $G_1$ in $\mathcal{S}$;
\item the diagonal is the unique common orbit of $A$ and $G$ in $\mathcal{S}\times \mathcal{S}$;
\item any $\sigma\in A$ with $A=\langle G,\sigma\rangle$ has a unique fixed point in $\mathcal{S}$;
\item any $\sigma\in A$ with $A=\langle G,\sigma\rangle$ has at most a fixed point in $\mathcal{S}$;
\item any $\sigma\in A$ with $A=\langle G,\sigma\rangle$ has at least a fixed point in $\mathcal{S}$.
\end{enumerate}
\end{lem}

Next, we give our third proof of Theorem~\ref{gtol} with a group-theoretic argument, which passes some property about the pair $(A,G)$ of the arithmetic monodromy group and the geometric monodromy group to the same property about the pair $(D,I)$ of the decomposition group and the inertia group over some prescribe rational place.

\begin{proof}[The third proof of Theorem~\ref{gtol}]
Let $L\colonequals \F_q(X)$ and $K\colonequals \F_q(Y)$ be the function fields of $X$ and $Y$, respectively. Thus $L/K$ is a finite separable extension of function fields over $\F_q$. Let $N$ be the Galois closure of $L/K$, and let $A$ and $G$ be the arithmetic and geometric monodromy groups of $L/K$, respectively.
By item (4) of Lemma~\ref{standard} the exceptionality of $f\colon X\to Y$ says that any $\sigma\in A$ with $A=\langle G,\sigma\rangle$ has a unique fixed point in $\mathcal{S}$, where $\mathcal{S}\colonequals \Hom_K(L,N)$ is the set of all $K$-algebra homomorphisms of $L$ into $N$.

Pick a place $w$ of $N$ which lies over the prescribed $\F_q$-rational place $u$ of $L$. Let $D$ and $I$ be the decomposition group and the inertia group of $N/K$ at $w$, respectively.
Since $DG=A$ the canonical homomorphism $D/I \to A/G$ is surjective, so that any $\sigma\in D$ with $D=\langle I,\sigma\rangle$ satisfies $A=\langle G,\sigma\rangle$.
By the above characterization of the exceptionality of $f$ in terms of $(A,G)$, we get the same property about the pair $(D,I)$, which says that any $\sigma\in D$ with $D=\langle I,\sigma\rangle$ has a unique fixed point in $\mathcal{S}$.

Let $N_w$, $L_u$, and $K_v$ be the completions of the function fields $N$, $L$, and $K$ with respect to the places $w$, $u$, and $v$, respectively. Thus $L_u$ is an intermediate field of the finite Galois extension $N_w/K_v$. We need to show the finite separable extension $L_u/K_v$ of local fields is exceptional. Let $\widehat N$ be the Galois closure in $N_w$ of $L_u/K_v$.
By item (10) of Theorem~\ref{nt-ram} the extension $L_u/K_v$ is exceptional if and only if any $\sigma\in \Gal(\widehat N/K_v)$ which induces the $q$-th power map on the residue field of $\widehat N$ has at most a fixed point in the set $\Hom_{K_v}(L_u,\widehat N)$ of all $K_v$-algebra homomorphisms of $L_u$ into $\widehat N$.
It is routine to check that any such $\sigma\in \Gal(\widehat N/K_v)$ can be lifted to some element $\widetilde\sigma\in D$ with $D=\langle I,\widetilde\sigma\rangle$, so that $\widetilde\sigma$ has a unique fixed point in $\Hom_K(L,N)$.
Since $\Hom_{K_v}(L_u,\widehat N)$ can be viewed canonically as a subset of $\Hom_K(L,N)$, it follows that $\sigma$ has at most a fixed point in $\Hom_{K_v}(L_u,\widehat N)$, so that the extension $L_u/K_v$ of local fields is exceptional.
\end{proof}

Let us conclude our third proof of Theorem~\ref{gcw} with another proof of Theorem~\ref{local}, which is more insightful but requires more knowledge of exceptional local field extensions developed in Section 4.

\begin{proof}[The second proof of Theorem~\ref{local}]
Write $n=mp^{\ell}$ where $p\colonequals \charp(\F_q)$ and $p\nmid m$. Then there exists an intermediate field $M$ of $L/K$ such that $[L:M]=p^{\ell}$ and $[M:K]=m$.
By Theorem~\ref{subext} the exceptionality of $L/K$ says that both $L/M$ and $M/K$ are exceptional. What we really need here is only that $M/K$ is exceptional, which by Theorem~\ref{coprime} says that $\gcd(m,q-1)=1$, or equivalently $\gcd(n,q-1)=1$.
\end{proof}

\end{document}